\documentclass{article}
\usepackage{amsmath, amssymb, amsthm, relsize, tikz-cd, graphicx}

\theoremstyle{plain}
\newtheorem{theorem}{Theorem}[section]
\newtheorem{lemma}[theorem]{Lemma}

\theoremstyle{definition}

\newcommand{\up}{\uparrow}

\begin{document}

\title{A Family of Bounded and Analytic Hyper-Operators}

\author{James David Nixon\\
	JmsNxn92@gmail.com\\}

\maketitle

\begin{abstract}
This is a summation of research done in the author's second and third year of undergraduate mathematics at The University of Toronto. As the previous details were largely scattered and disorganized; the author decided to rewrite the cumulative research. The goal of this paper is to construct a family of analytic functions $\alpha \uparrow^n z : (1,e^{1/e}) \times \mathbb{C}_{\Re(z) > 0} \to \mathbb{C}_{\Re(z) > 0}$ using methods from fractional calculus. This family satisfies the hyper-operator chain, $\alpha \uparrow^{n-1} \alpha \uparrow^n z = \alpha \uparrow^n (z+1)$; with the initial condition $\alpha \uparrow^0 z = \alpha \cdot z$.
\end{abstract}

\emph{Keywords:} Complex Analysis; Fractional Calculus; Complex Dynamics.\\

\emph{2010 Mathematics Subject Classification:} 30D05; 30B50; 37F10; 39B12; 39B32\\

\section{Introduction}\label{sec1}
\setcounter{equation}{0}

This paper is, in many ways, a rewrite of what I had discovered and written of six years past. As the previous work was, shoddy and a tad bit unprofessional--but largely the mathematics was correct (subtracting a few caveats); I decided to rewrite. This paper centers on an alternative method of expressing local fixed point iterations. This method largely avoids what one would traditionally call iteration theory, in favour of a language of integral transforms.\\

As to this, the author will frame his language using The Riemann-Liouville Differintegral. As to what work by the author which remains introducing this concept, we reference our research from a similar time frame, in \textit{On The Indefinite Sum In Fractional Calculus} \cite{Nix}. But, we shan't use anything from this paper; and the paper is written in a rather rigid style; and is a tad too lofty for the author's taste now.

As to this, we shall reintroduce the fractional derivative as we use it; and state the fundamental lemmas. Much of this work is standard in a study of Mellin Transforms; or any work pertaining to Ramanujan. The keystone of this arch, as it was before, is still Ramanujan's Master Theorem. However, we choose to write this theorem as an interpolation of derivatives. Which, henceforth the fractional derivative becomes the main center of focus, as opposed to the Mellin transform.

We needn't introduce the deepest depths of fractional calculus, we just need some of its basics. And next to it; we need some basics from Milnor's treatment of complex dynamics. With this, we only need the construction of the Schr\"{o}der function, and some normality theorems on functional iterates. 

The focus of this paper will be much more expository than expansive. There is much in this paper we could generalize. The purpose is solely to construct a chain of hyper-operators, and nothing else. We don't delve into the nitty gritty of how else this representation works. For such, we argue by example.

For that's what this paper is, an argument for a representation of complex iterations. This is done largely off-hand; and again, not to the depth that it can be done. The author has argued for this representation largely outside of publications, but it still exists in a nether-world. It isn't particularly part of mainstream iteration theory.\\

What we'll construct is a fascination, rather than a necessity. This is nothing important to the world, but is interesting to mathematicians; at least those I've met. That is the idea of a hyper-operator. Something similar will be done in this paper as was done in \textit{Hyper-operations by Unconventional Means} \cite{Nix2}. Excepting, in that paper we considered $e \up^n x$ for real $x$; and we only constructed a smooth solution. We won't use any of the tools from that paper though; the construction is much more sophisticated there; it is much more intuitive here, if one can grasp the fractional derivative as something no more complicated than the Laplace, or Fourier Transform.

As a formal sequence, we can call a hyper-operator chain, a sequence of functions $F_n$ such that,

\[
F_{n}(F_{n+1}(z)) = F_{n+1}(z+1)\\
\]

And we can require these things to be holomorphic, or smooth, or however. We are interested in a chain which satisfies the initial conditions,

\begin{eqnarray*}
F_0(z) &=& \alpha \cdot z\,\,\text{for}\,\,1 < \alpha < e^{1/e}\\
F_1(z) &=& \alpha^z\\
F_n(1) &=& \alpha\\
\end{eqnarray*}

This produces what is traditionally called a hyper-operator with base $\alpha$. We require these functions are holomorphic, and furthermore, that they satisfy specific normality conditions (we'll get to that later).

This paper is designed to promote an algorithm for producing this sequence. This algorithm is simply an iteration of integral transforms. Which is to say, we perform an integral transform $n$ times to produce our function $F_n$--minusing a couple of technical details. This is very similar to a Picard iteration; however, slightly more involved.

\section{The fractional derivative of interest}\label{sec2}
\setcounter{equation}{0}

There are quite a few fractional derivatives which exist in the world. But, the central difference being the base-point we assign; there is very little difference between them. Forgoing domains, and methods, much of the theory revolves on Cauchy's iterated integral formula--this can be found in most real analysis textbooks. This is written,

\[
\frac{d^{-n}}{dx^{-n}}f(x) = \frac{1}{\Gamma(n)}\int_a^x (x-t)^{n-1}f(t)\,dt\\
\]

Which accounts for $n$ integrals, with a base point of $a$, which is,

\[
\frac{d^{-n}}{dx^{-n}}f(x) = \int_a^x \int_a^{x_{n-1}}...\int_a^{x_1} f(t)\,dt dx_1...dx_{n-1}\\
\]

The specific case we want, is the integration which fixes the exponential $e^x$, which is given by $a = -\infty$. Which is,

\[
\int_{-\infty}^x e^t \,dt = e^x\\
\]

To derive a fractional derivative, we just set $n$ to complex values, and deride where it converges. For a fuller treatment on the history; and results extraneous to our study, we point to the reference \textit{The Fractional Calculus}, by Keith B. Oldham and Jerome Spanier \cite{Old}. But forgoing a deep dive into fractional calculus; we need next to nothing from this book, but it may ease the pace we go at.

Massaging our formula, and taking the short route, our fractional derivative is,

\[
\frac{d^{-z}}{dx^{-z}} f(x) = \frac{1}{\Gamma(z)}\int_0^\infty f(x-t)t^{z-1}\,dt\\
\]

This is nothing more than a modified Mellin transform. However, we don't want this to be our entire definition. As we're dealing with holomorphy, we're instead going to consider entire functions $f(w)$ for $w \in \mathbb{C}$; and we want a type of convergence in the complex plane. The reason we make the following jump is to avoid certain anomalous cases. If we take,

\[
\frac{d^{-z}}{dx^{-z}} \sin(x) = \sin(x-z\pi/2)
\]

This converges for $0 < \Re(z) < 1$; however this function will produce many anomalies in further studies. The trouble is actually typical of $\sin$ and functions like $\sin$; which we want to eliminate from our study. The function $\sin(x) \not \to 0$ as $x \to \infty$, but further $\sin(w)$ has exponential growth as $w \to \infty$; the Mellin transform won't make sense.

We want the Mellin transform to work in neighborhoods of the variable $w$; as to that, we have to consider arbitrary arcs $\gamma$; not just along the line $[0,\infty]$. To begin, call,

\[
S_\theta = \{w \in \mathbb{C}\,|\,|\arg(w)| < \theta\}
\]

And assume for all arcs $\gamma \in S_\theta$, such that $\gamma(0) = 0$ and $\gamma(\infty) = \infty$; that,

\[
\int_\gamma |f(t)|t^{\sigma - 1}\,dt < \infty\\
\]

For some $\sigma \in \mathbb{R}^+$. Then the differintegral of $f$ is written, for $0 < \Re(z) < \sigma$,

\[
\frac{d^{-z}}{dw^{-z}} f(w) = \frac{1}{\Gamma(z)} \int_\gamma f(w-y)y^{z-1}\,dy\\
\]

Which is equivalent for all paths $\gamma$ by Cauchy's integral theorem. Considering the differintegral as this, helps us treat later theorems with greater ease. This also allows us to frame an isomorphism; which was used in \cite{Nix}; but we'll reintroduce here.

Call,

\begin{align*}
\mathbb{S}_\theta &= \{f : \mathbb{C} \to \mathbb{C},\,f\,\text{is holomorphic}\,\text{s.t.}\, \int_\gamma|f(y)|\,dy < \infty\,\\
&\text{where}\,\gamma \subset S_\theta,\,\gamma(0) = 0,\,\gamma(\infty) = \infty\}\\
\end{align*}

Then there exists a correspondence between functions $F(z)$ holomorphic on $\mathbb{C}_{\Re(z) > -1}$ such that $|\Gamma(-z)F(z)| \le C_\delta e^{(\delta-\theta)|\Im(z)|}$ for all $\delta > 0$ where $C_\delta \in \mathbb{R}^+$ while $z \not \in \mathbb{R}$. This is the central theorem we present in this section.

It's interesting to follow the history of this result, as it is to follow the math. The first observation is due to Euler. As this analytic continuation is first discovered by him, we can call it Euler's Analytic Continuation Theorem.

\begin{theorem}[Euler's Analytic Continuation Theorem]\label{thmEUL}
Suppose that $f \in \mathbb{S}_\theta$; then the function,

\[
\frac{d^{-z}}{dw^{-z}}f(w) = \frac{1}{\Gamma(z)}\int_\gamma f(w-y)y^{z-1}\,dy\,\,\text{for}\,\,0 < \Re(z) < 1\\
\]

Can be analytically continued to,

\[
\Gamma(z)\frac{d^{-z}}{dw^{-z}}f(w) = \sum_{n=0}^\infty f^{(n)}(w)\frac{(-\gamma(1))^n}{n!(n+z)} + \int_{\gamma[1,\infty)} f(w-y)y^{z-1}\,dy\,\,\text{for}\,\,\Re(z) < 1\\
\]
\end{theorem}

\begin{proof}
Break the integral into two pieces $\int_0^\infty = \int_0^1 + \int_1^\infty$. The integral over $\int_0^1$ converges to the series,

\[
\int_{\gamma[0,1]} f(w-y)y^{z-1}\,dy = \sum_{n=0}^\infty \frac{f^{(n)}(w)(-1)^n}{n!}\int_0^{\gamma(1)} y^{n+z-1}\,dy\\ 
\]

Which converges as the first sum, which is holomorphic where ever $z\neq -n$, but at $z=-n$ we have a simple pole, and $\frac{1}{\Gamma}$ has a simple zero; where there limit is $f^{(n)}(w)$. The second integral converges for $\Re(z) < 1$.
\end{proof}

We attribute this result to Euler mostly by way of learning of this technique through him. As far as the author knows, this is his own result; he's yet to see it published elsewhere. There exists a paper which did much of the similar analysis; but lacked this language \cite{ZAG}. But he learned it by way of Euler; who used it as a technique to analytically continue the $\Gamma$-function. Which is written,

\[
\Gamma(z) = \sum_{n=0}^\infty \frac{(-1)^n}{n!(n+z)} + \int_1^\infty e^{-t}t^{z-1}\,dt\\
\]

Which is valid in the complex plane as a meromorphic function. One can instantly see the similarity between our analytic continuation, and the $\Gamma$-function. The $\Gamma$-function sort of acts as a normalizer. But, going even further; it aids perfectly in interpolation theory as well. 

Before going on, we'd like to add a result by Stirling. We will not prove this result, but reference its appearance in Reinhold Remmert's \textit{Classical Topics in Complex Function Theory} \cite{Rem}. Remmert does a fantastic job of providing not only the theorem, but an in depth examination of the error term. This result is written slightly differently than one may be used to, as it deals with the $\Gamma$-function for complex values.

\begin{theorem}[Stirling's Theorem]
The function $\Gamma(z)$ satisfies the asymptotic equation,

\[
\Gamma(z) \sim \sqrt{2\pi} z^{z-\frac{1}{2}}e^{-z}\,\,\text{as}\,\,|z|\to\infty\,\,\text{while}\,\,|\arg(z)| < \pi\\
\]

And,

\[
|\Gamma(x+iy)| \sim \sqrt{2\pi}|y|^{x-\frac{1}{2}} e^{-\pi |y|/2}\\
\]
\end{theorem}

The following result is attributed to Ramanujan; a full proof can be found in \cite{RAM}; we'll only write a small sketch here. We'll also restrict it to functions of our interest; and change the language sufficiently.

\begin{theorem}[Ramanujan's Master Theorem]\label{thmRAM}
Suppose that $H(z)$ is holomorphic for $\Re(z) > -1$. Suppose it satisfies the bounds,

\[
|\Gamma(-z)H(z)| \le C_\delta e^{(\delta-\theta) |\Im(z)|}\,\,\text{for all}\,\,\delta>0\,\,\text{and}\,\,C_\delta\in \mathbb{R}^+,\,\text{while} \,|\arg(z+1)| \le \pi/2\\
\]

Then the function,

\[
f(w) = \sum_{n=0}^\infty H(n)\frac{w^n}{n!}\\
\]

satisfies,

\[
\frac{d^z}{dw^z}\Big{|}_{w=0}f(w) = H(z)\\
\]

Or as a differintegration,

\[
\frac{d^z}{dw^z} f(w) = \sum_{n=0}^\infty H(z+n)\frac{w^n}{n!}\\
\]
\end{theorem}

\begin{proof}
Take the function $H(z)$ and consider, for $0<\sigma<1-\Re(z)$,

\[
f(w,z) = \frac{1}{2\pi i} \int_{\sigma- i\infty}^{\sigma + i\infty} \Gamma(s)H(z-s)(-w)^{-s}\,ds\\
\]

By the asymptotics of the $\Gamma$-function (Stirling's Theorem), we get that this integral converges $w \in S_\theta$. By Cauchy's Residue Theorem, and a quick estimation using Stirling, we get that this equals the sum of all the residues of the singularities in the left half plane. 

\[
f(w,z) = \sum_{n=0}^\infty \text{Res}(s=-n, \Gamma(s)H(z-s)(-w)^{-s})
\]

But,

\[
\text{Res}(s=-n, \Gamma(s)H(z-s)(-w)^{-s}) = H(z+n)\frac{w^n}{n!}\\
\]

On our other foot though; this is The Inverse Mellin Transform; so therefore,

\[
\int_0^\infty f(-y,z)y^{\mu-1}\,dy = \Gamma(\mu) H(\mu + z)\\ 
\]

Which for $z=0$ is the statement of Ramanujan's Master Theorem.
\end{proof}

From this theorem, we write the central thesis of the intuition we'll use in later sections. This can be explained as a commutative diagram. There needs one more piece though, which can be found by introducing the space,

\begin{align*}
\mathbb{E}_\theta = &\{F\,\text{is holomorphic for}\,\Re(z) > -1\,\,\text{s.t.}\\
&\,\,\limsup_{n\to\infty}\left|F(z+n) \right |^{1/n} =0\\
&\,\,|\Gamma(z)F(-z)|\le C_\delta e^{\delta-\theta|\Im(z)|}\\
&\text{for all }\delta>0\,\,\text{where}\,C_\delta\in \mathbb{R}^+\,\text{for}\,\pi/2 \le |\arg(z)| < \pi\}\\
\end{align*}

Then the commutative diagrams are,

\begin{center}
\begin{tikzcd}
\mathbb{S}_\theta \arrow[rr, "\frac{d}{dw}\cdot"] \arrow[dd, "\frac{d^z}{dw^z}\Big{|}_{w=0} \cdot"] &  & \mathbb{S}_\theta \arrow[dd, "\frac{d^z}{dw^z}\Big{|}_{w=0} \cdot"] \\
                                                                                       &  &                                                                 \\
\mathbb{E}_\theta \arrow[rr, "z \mapsto z+1"]                                                       &  & \mathbb{E}_\theta                                                         
\end{tikzcd}
\end{center}

And,

\begin{center}
\begin{tikzcd}
\mathbb{E}_\theta \arrow[dd, "\sum_{n=0}^\infty F(n)\frac{w^n}{n!}"] \arrow[rr, "z \mapsto z+\mu"] &  & \mathbb{E}_\theta \arrow[dd, "\sum_{n=0}^\infty F(n)\frac{w^n}{n!}"] \\
                                                                                      &  &                                                                 \\
\mathbb{S}_\theta \arrow[rr, "\frac{d^\mu}{dw^\mu}"]                                               &  & \mathbb{S}_\theta                                  
\end{tikzcd}
\end{center}

Where these are isomorphic relationships. The only thing left to really prove is that,

\[
f \in \mathbb{S}_\theta \Leftrightarrow \frac{d^z}{dw^z}\Big{|}_{w=0} f \in \mathbb{E}_\theta\\
\]

This is solved by looking at contour integrals. We write our first real non-trivial theorem, which isn't referenced elsewhere.

\begin{theorem}\label{thmDIFFBND}
Let $f \in \mathbb{S}_\theta$, then,

\[
\frac{d^z}{dw^z}\Big{|}_{w=0}f(w) = F(z) \in \mathbb{E}_\theta\\
\]

And conversely.
\end{theorem}

\begin{proof}
The integral,

\[
F(-z) = \frac{1}{\Gamma(z)}\int_\gamma f(-y)y^{z-1}\,dy\\
\]

Converges for the arc $\gamma = [0,e^{i(\theta-\delta)}\infty)$. Subbing this in we get,

\[
F(-z) = \frac{e^{i(\theta-\delta)z}}{\Gamma(z)} \int_0^\infty f(-e^{i(\theta-\delta)}x)x^{z-1}\,dx\\
\]

The integral is bounded as,

\[
|F(-z)| \le \left |\frac{e^{i(\theta-\delta)z}}{\Gamma(z)}\right|\int_0^\infty |f(-e^{i(\theta-\delta)}x)|\,dx\\
\]

In the upper half plane (resp. lower half plane) the term $\left |e^{i(\theta-\delta) z}\right| \le Ce^{(\delta-\theta)|\Im(z)|}$. This relationship can be derived throughout by differentiating $f$, and shifting $F$. Which says $F \in \mathbb{E}_\theta$. Or we can observe that,

\[
\Gamma(-z)F(z) = \sum_{n=0}^\infty f^{(n)}(w)\frac{(-1)^n}{n!(n-z)} + \int_1^\infty f(w-y)y^{-z-1}\,dy \to 0\,\, \text{as}\,\,\Re(z) \to \infty\\
\]

If $F$ in $\mathbb{E}_\theta$, then using similar asymptotics,

\[
f(w) = \frac{1}{2 \pi i} \int_{\sigma - i\infty}^{\sigma + i \infty} \Gamma(z)F(-z)(-w)^{-z}\,dw\\
\]

Is valid for $w \in S_\theta$. And by Mellin's Inversion Theorem, we're done.
\end{proof}

For this section we can introduce the central theorem, which is just an isomorphism. We can aptly call this The Differintegral Isomorphism.

\begin{theorem}[The Differintegral Isomorphism]\label{thmDIFFISO}

For all $f \in \mathbb{S}_\theta$ there exists a unique $F \in \mathbb{E}_\theta$. And for all $F \in \mathbb{E}_\theta$ there exists a unique $f \in \mathbb{S}_\theta$. Upon which,

\[
f(w) = \sum_{n=0}^\infty F(n)\frac{w^n}{n!}\\
\]

And,

\[
F(z) = \frac{d^z}{dw^z}f(w)\Big|_{w=0}
\]

\end{theorem}

This theorem will be the central focus of the rest of this paper. It means, in plain english; for every differintegrable entire function $f$, there is a well behaved differintegral $F$; and vice versa. This will make more sense as we progress.

\section{Schr\"{o}der's equation in the nicest possible conditions}

The purpose of this section is to describe iterates of a function $\phi: \mathcal{G} \to \mathcal{G}$. Specifically where there is a unique attracting fixed point $\xi_0 \in \mathcal{G}$ where $\phi(\xi_0) = \xi_0$. And additionally that the multiplier of said fixed point is $0 < \phi'(\xi_0) < 1$. Then, we want to construct a function $\phi^{\circ z}(\xi)$. This is doable, and has been done to its exhaustion.

We will briefly follow John Milnor's treatment of local fixed point theory (particularly geometrically attracting fixed points) from \textit{Dynamics in One Complex Variable} \cite{Mil}. We will not need much from this book; but we'll borrow its terminology, and some of its results.

The knick, the knack, and all that slaps back; is that we can write the usual iteration methods using integral transforms. In order to do this; we have to keep the derivative of $\phi$ non-zero; but all else falls in line. This is the first isomorphism we'll explain; which is Schroder's function. There exists a function,

\[
\Psi(\phi(\xi)) = \phi'(\xi_0) \Psi(\xi)\\
\]

Where $\Psi$ is holomorphic in a neighborhood of $\xi_0$ at least; and sends $\Psi(\xi_0) = 0$. Call a neighborhood $\mathcal{U}$ in which we define the iteration,

\[
\phi^{\circ z}(\xi) = \Psi^{-1}(\phi'(\xi_0)^z \Psi(\xi)) : \mathbb{C}_{\Re(z) > 0} \times \mathcal{U} \to \mathcal{U}\\
\]

Now, since this function is periodic in $z$ with period $2\pi i / \log(\phi'(\xi_0))$; we know this function satisfies the bound,

\[
||\phi^{\circ z}(\xi)||_{\Re(z) > 0, \mathcal{U}} \le M\\
\]

For some $M$. This implies our function $\phi^{\circ z}(\xi) \in \mathbb{E}_\theta$ for all $\theta \le \pi/2$. (Recalling that the Gamma function has decay like $e^{-\pi/2|\Im(z)|}$.) Then, we have a function,

\[
\vartheta(w,\xi) = \sum_{n=0}^\infty \phi^{\circ n+1}(\xi)\frac{w^n}{n!}\\
\]

Which satisfies,

\[
\phi^{\circ z}(\xi) = \frac{d^{z-1}}{dw^{z-1}}\Big{|}_{w=0} \vartheta(w,\xi)\\
\]

And the functional equation,

\[
\vartheta(w,\phi(\xi)) = \frac{d}{dw}\vartheta(w,\xi)\\
\]

The importance of this being that $\vartheta$ makes no mention of the Schr\"{o}der function, or its inverse. As that can be exhausting to compute; this is very helpful. All we need to construct the complex iteration is the natural iterates of $\phi$.

If we call $\mathcal{A}$ the basin of attraction about the fixed point $\xi_0$; which is,

\[
\xi \in \mathcal{A} \,\Rightarrow\,\lim_{n\to\infty}\phi^{\circ n}(\xi) = \xi_0\\
\]

The surely the function,

\[
\vartheta(w,\xi) : \mathbb{C} \times \mathcal{A} \to \mathbb{C}\\
\]

We're going to restrict this domain a bit; and take the maximal connected component of $\mathcal{A}$ which contains $\xi_0$. This is known as the immediate basin of attraction. We'll denote this $\mathcal{A}_0$.

Then the central focus of this section is to show that,

\[
\phi^{\circ z}(\xi) = \frac{d^{z-1}}{dw^{z-1}}\Big{|}_{w=0}\vartheta(w,\xi) : \mathbb{C}_{\Re(z) > 0} \times \mathcal{A}_0 \to \mathcal{A}_0\\
\]

Which is why we've made the requirement that $\phi'(\xi) \neq 0$; that the derivative is non-singular. There are two ways of proving this result, and both are rather technical. We can talk about Schr\"{o}der's function more; or we can restrict our conversation to integral transforms.

However, both display the same central idea. That is, to pull back our solution with $\phi^{-1}$. Essentially this is done with the implicit function theorem,

\[
\phi(y) = \phi^{\circ z+1}(\xi)\\
\]

Then,

\[
y = \phi^{\circ z}(\xi)\\
\]

These values always exist, and $\phi$ is non singular; and so this constructs a holomorphic function on $\mathcal{A}_0$. This, again, will be periodic, and therefore bounded and within $\mathbb{E}_\theta$. We can write this as a theorem.

\begin{theorem}[Fractional Calculus Representation Theorem]\label{thmREPFRC}
Let $\mathcal{G} \subseteq \mathbb{C}$ be a domain. Assume that $\phi : \mathcal{G} \to \mathcal{G}$ and for some $\xi_0 \in \mathcal{G}$ we have $\phi(\xi_0) = \xi_0$ and $0 < \phi'(\xi_0) < 1$. Additionally, assume that $\phi'(\xi) \neq 0$. Then on the immediate basin of attraction $\mathcal{A}_0$, there exists a fractional iteration,

\[
\phi^{\circ z}(\xi) : \mathbb{C}_{\Re(z) > 0} \times \mathcal{A}_0 \to \mathcal{A}_0\\
\]

Which can be represented as,

\[
\phi^{\circ z}(\xi) = \frac{d^{z-1}}{dw^{z-1}}\Big{|}_{w=0}\vartheta(w,\xi)\\
\]

For,

\[
\vartheta(w,\xi) = \sum_{n=0}^\infty \phi^{\circ n+1}(\xi) \frac{w^n}{n!}\\
\]
\end{theorem}

\section{The Tetration Function}

This section is intended to create a base-step in our inductive proof to come. This is where we start to construct a non-trivial result. We suggest the reader google the term `tetration' for a better understanding of what it is. There exists a scant amount of literature on the subject; and of this, we assume the reader is somewhat familiar with the idea. We want to construct a tetration function,

\[
\alpha \up \up z\\
\]

For $1 < \alpha < \eta = e^{1/e}$ and for $\Re(z) > 0$; which maps to $\Re(z) > 0$. The only difficult part of this section is ensuring that this tetration maps to the right half plane. Start by considering the function,

\[
\phi(\xi) = \alpha^\xi\\
\]

Which has a fixed point at $1 < \omega < e$ where $\alpha = \omega^{1/\omega}$. The multiplier of this fixed point is,

\[
\phi'(\omega) = \log(\omega)\\
\]

Which is in the interval $(0,1)$. Additionally, this function has a non-zero derivative. Therefore, it satisfies all the requirements of The Fractional Calculus Representation Theorem \ref{thmREPFRC}. So in the immediate basin $\mathcal{A}_0$ about $\omega$ we have,

\[
\phi^{\circ z}(\xi) : \mathbb{C}_{\Re(z) > 0} \times \mathcal{A}_0 \to \mathcal{A}_0\\
\]

It's not hard to notice that $\phi^{\circ n}(1) = \alpha^{\displaystyle \alpha^{...\alpha}} \to \omega$. This result also dates back to Euler. So if we define,

\[
\vartheta(w)= \sum_{n=0}^\infty \phi^{\circ n+1}(1) \frac{w^n}{n!} = \sum_{n=0}^\infty \left( \alpha^{\displaystyle\alpha^{...n+1\,\text{times}...\alpha}}\right)\frac{w^n}{n!}\\
\]

Then,

\[
\alpha \up \up z = \frac{d^{z-1}}{dw^{z-1}}\Big{|}_{w=0} \vartheta(w)\\
\]

Now, in addition to this; the operator $z \mapsto z+1$ is a contraction mapping. And it contracts to the fixed point at $z=\omega$. In order to understand this, we have to know that $\alpha \up \up z$ is injective.

\begin{lemma}[The Injectivity Lemma]\label{thmINJ}
The function $\phi^{\circ z}(\xi)$ from Theorem \ref{thmREPFRC} is injective in $z$ on the strip $0 \le \Im(z) < 2\pi/|\log(\phi'(\xi_0))|$.
\end{lemma}

\begin{proof}
Suppose that,

\[
\phi^{\circ z_1}(\xi) = \phi^{\circ z_2}(\xi)\\
\]

Then,

\[
\phi^{\circ z_1 + n} = \phi^{\circ n}(\phi^{\circ z_1}) = \phi^{\circ n}(\phi^{\circ z_2}) = \phi^{\circ z_2 +n}(\xi)\\
\]

Therefore,

\[
\phi^{\circ z_1 + z}(\xi) = \phi^{\circ z_2 + z}(\xi)\\
\]

By expanding them in their integral representation. Which is, if $F,G \in \mathbb{E}_\theta$ and $F\Big{|}_{\mathbb{N}} = G \Big{|}_{\mathbb{N}}$; then $F=G$. Now, $z_1 - z_2$ must be a period of $\phi^{\circ z}$; but this is impossible in the strip we chose.
\end{proof}

With this, we can confidently say the contraction lemma.

\begin{lemma}[The Contraction Lemma]\label{lmaCNTRC}
The function $\phi^{\circ z}(\xi)$ from Theorem \ref{thmREPFRC} satisfies,

\[
|\phi^{\circ z}(\xi) - \xi_0| \le |\xi - \xi_0|\\
\]
\end{lemma}

\begin{proof}
Since our function is injective, we can define a coordinate,

\[
\zeta = \phi^{\circ z}(\xi)\\
\]

Additionally, there is a domain,

\[
\zeta \in \mathcal{U} = \{\zeta \in \mathcal{A}_0\,|\,\zeta = \phi^{\circ z}(\xi)\,\,\text{for}\,\Re(z)>0\}\\
\]

For any $\Re \mu > 0$ we know that $\lim_{n\to\infty} \phi^{\circ \mu n}(\zeta) \to \xi_0$. But additionally; this is a contraction on $\mathcal{U}$. Therefore, in this coordinate,

\[
|\phi^{\circ \mu}(\zeta) - \xi_0| \le |\zeta - \xi_0|\\
\]

Which is the theorem statement.
\end{proof}

From this, we see the finish line in our sights. It's a little off kilter, but it's there. We must have that,

\[
|\alpha \up \up z - \omega| \le |1-\omega|\\
\]

And the domain,

\[
\mathcal{H} = \{\zeta \in \mathbb{C}\,|\,|\zeta - \omega| \le |1-\omega|\}\subset\mathbb{C}_{\Re(z) > 0}\\
\]

Therefore, we have the base-step of our induction.

\begin{theorem}[Bounded Analytic Tetration Theorem]\label{thmBNDTET}
The function,

\[
\alpha\up\up z : (1,\eta) \times \mathbb{C}_{\Re(z) > 0} \to \mathbb{C}_{\Re(z) > 0}\\
\]
\end{theorem}

We use this theorem, and everything we've built up, to define a sequence of results. Or rather, an algorithm for constructing hyper-operations. In many ways, much of this work is pretty leveled and exists in previous work. The main, je ne sais quoi, of this method, is of designing an algorithm using integral transforms. 

\section{Bounded and Analytic Hyper-operators}\label{sec5}
\setcounter{equation}{0}

This section will be a quick proof by induction. For that reason, assume we've constructed a hyper-operator,

\[
\alpha \up^{n-1} z : (1,\eta)\times\mathbb{C}_{\Re(z) > 0} \to \mathbb{C}_{\Re(z) > 0}\\
\]

And we'll use this to construct a hyper-operator at $n$. Since we know that $\alpha \up^{n-1} \mathbb{R}^+ \subset (1,e)$, at least. We know there's a fixed point on the real line; and additionally that its attracting on the real line. Call this number $\omega_{n}$. We know the multiplier of this fixed point, $\lambda$, satisfies at least, $0 \le \lambda \le 1$. 

Since $\alpha \up^{n-1} z$ maps a simply connected domain to itself, and has a fixed point, and is not a Linear Fractional Transformation, it must have a multiplier $\lambda < 1$. Assume, additionally that $\alpha \up^{n-1}(z)$ has a non-vanishing derivative; then $0 < \lambda < 1$. Therefore it has a geometrically attracting fixed point and we can construct,

\[
\alpha \up^{n} z = \phi^{\circ z}(1)\\
\]

Where, $\phi(\xi) = \alpha \up^{n-1} \xi$. We additionally know by The Injective Lemma \ref{thmINJ} and The Contraction Lemma \ref{lmaCNTRC}; that we must have,

\[
\alpha \up^{n} z : (1,\eta) \times \mathbb{C}_{\Re(z) > 0} \to \mathbb{C}_{\Re(z) > 0}\\
\]

So, all that's left is to show its derivative is non-vanishing. We provide this as a quick lemma.

\begin{lemma}
The function $\phi^{\circ z}(\xi)$ from Theorem \ref{thmREPFRC}, satisfies,

\[
\frac{d}{dz}\phi^{\circ z}(\xi) \neq 0\\
\]
\end{lemma}

\begin{proof}
Assume that,

\[
\frac{d}{dz}\phi^{\circ z_0} = 0\\
\]

Then,

\[
\frac{d}{dz} \phi^{\circ n}(\phi^{\circ z_0})= \frac{d}{dz} \phi^{\circ z_0 + n} = 0\\
\]

The function,

\[
H(z) = \frac{d}{dz}\phi^{\circ z}\\
\]

Is periodic, with purely imaginary period; and it tends to $0$ as $\Re(z) \to \infty$. Therefore,

\[
||H(z)||_{\mathbb{C}_{\Re(z)>0}} \le M\\
\]

For some $M$. Therefore $H \in \mathbb{E}_\theta$ for $\theta \le \pi/2$. Since,

\[
H(z_0 + n) = 0\\
\]

We must have, by using The Differintegral Isomorphism \ref{thmDIFFISO},

\[
H(z_0 + z) = 0\\
\]

Which is a contradiction.
\end{proof}

Therefore our function $\alpha \up^n z$ has a non-vanishing derivative. With that, we state the algorithm.

\begin{theorem}[The Bounded Analytic Hyper-operator Algorithm]\label{ALGBND}
Define a sequence of functions $\alpha \up^n z$, starting with $\alpha \up^0 z = \alpha \cdot z$; and a sequence of functions,

\[
\vartheta_n(w,\xi) = \sum_{k=0}^\infty \left(\alpha \up^{n-1} \alpha \up^{n-1}\cdots(k+1\,\text{times})\cdots\up^{n-1}\alpha\right)\frac{w^k}{k!}\\
\]

Where,

\[
\alpha \up^n z = \frac{d^{z-1}}{dw^{z-1}}\Big{|}_{w=0} \vartheta_n(w)\\
\]

Then,

\begin{enumerate}
    \item $\alpha \up^n z : (1,\eta)\times \mathbb{C}_{\Re(z) > 0} \to \mathbb{C}_{\Re(z) > 0}$ and is analytic.
    \item $\alpha \up^n x : (1,\eta) \times \mathbb{R}^+ \to (1,e)$ and is monotone.
    \item $\alpha \up^n z$ has a purely imaginary period, and is injective in a strip of width of said period.
    \item $\frac{d}{dz} \alpha \up^n z \neq 0$.
    \item The functional equation $\alpha \up^{n-1} \alpha \up^n z = \alpha \up^n (z+1)$ is always satisfied. And $\alpha \up^n 1 = \alpha$.
\end{enumerate}
\end{theorem}

\section{Generalizations}\label{sec6}
\setcounter{equation}{0}

There are two generalizations the author would like to briefly add. The first being more closely related to complex dynamics, and the second more closely related to the general fractional calculus theory. We will only briefly justify these things.

If we limit our chain of hyper-operators $\alpha \to \eta$, we have no problem whatsoever. The only anomaly which arises is for the tetration case. This is because we have a neutral fixed point at $e$, rather than a geometrically attracting one. Which is,

\[
\eta \up \up z = \frac{d^{z-1}}{dw^{z-1}}\Big{|}_{w=0} \sum_{k=0}^\infty \exp_\eta^{\circ k+1}(1)\frac{w^k}{k!}\\
\]

The proof of this is different than what we've done here though. One has to construct an iteration in a different manner. This is typically framed as finding an Abel function of $e^\xi -1$; which looks like $\eta^\xi$ under an appropriate conjugation $\xi \mapsto \xi/e - 1$. Then, going by inspection the iterate is an element of $\mathbb{E}_\theta$ for $\theta < \pi/2$, and everything follows in the same manner.

For general neutral fixed points a similar form arises; but we have to be more careful. Abel functions can be a bit more precarious compared to a nice standard iteration. As to this, one should tread carefully when making the correspondence.\\

The second point is as to the benefit of this method in more general scenarios. Consider $A$ an $n\times n$ diagonalizable matrix. And consider the function,

\[
\vartheta(w) = e^{Aw} = \sum_{k=0}^\infty A^k \frac{w^k}{k!}\\
\]

Let's list the eigenvalues $\lambda_1,...,\lambda_n$. Assume there is some sector $S_\theta$ such that,

\[
e^{-\lambda_j w} \to 0\\
\]

Then, the matrix $A^z$ is given by,

\[
A^z = \frac{d^{z}}{dw^{z}}\Big{|}_{w=0} e^{Aw}\\
\]

With eigenvalues $\lambda_1^z,...,\lambda_n^z$. This is pretty easy to prove, and is left to the reader. The reason this becomes important; we can view this as iterating operators. If we take $n = \infty$; and talk about diagonalizable matrices here; we're on the verge of talking about Hilbert Spaces. And in this manner, one can think of ways of iterating operators on Hilbert spaces using fractional calculus.

This, as it leaves the perview of the author, is left to the reader.

\section{In Conclusion}

We thank the reader for their time, and their patience. We hope we have shed light on much of the work the author had done what seems a life-time ago. And what was written, in perhaps too cryptic a manner to be fully appreciated.

\end{document}